\documentclass{article}
\usepackage[english]{babel}
\usepackage{amsmath,amssymb,graphicx,bbm,latexsym,tikz-cd}

\newcommand{\longhookrightarrow}{{\lhook\joinrel\relbar\joinrel\rightarrow}}

\newcommand{\nobracket}{}
\newcommand{\tmname}[1]{\textsc{#1}}
\newcommand{\tmop}[1]{\ensuremath{\operatorname{#1}}}
\newcommand{\tmrsub}[1]{\ensuremath{_{\textrm{#1}}}}
\newenvironment{proof}{\noindent\textbf{Proof\ }}{\hspace*{\fill}$\Box$\medskip}
\newtheorem{corollary}{Corollary}
\newtheorem{lemma}{Lemma}
\newtheorem{proposition}{Proposition}
\newtheorem{theorem}{Theorem}

\begin{document}
\title{An elementary proof of Poincare Duality with local coefficients}
\author{Fang Sun}
\maketitle

\section{Introduction}

The statement and proof of the Poincare Duality for (possibly noncompact) orientable manifolds
without boundary are abound. The duality has a version for possibly
non-orientable manifolds using local coefficients. Several proofs of this result can be found in the literature (e.g. [Sp]). Yet these proof involves either sheaf theory or something equivalent and thus far from elementary. This
note is written in the attempt to provide a record of this generalized Poincare Duality available to the general audience.

Readers are assumed to be familiar with basics of the theory of orientation of
manifolds (the material of [Ha] Section 3.3 or any equivalent) and the
theory of homology with local coefficients (e.g. [Wh] Chapter VI,
Section 1,2,3).

Notations of [Ha] and [Wh] will be borrowed and used in this note.

Throughout our discussion, $M$ will be a manifold of dimension n.

Let $R$ be a ring with identity such that $2 \cdot \tmop{id} \neq 0$.

\section{Preliminaries}

\subsection{The Orientation Bundle}

Let $M_R = \underset{x \in M}{\sqcup} H_n (M|x ; R)$, topologized as follows:

For any chart $\varphi : W \rightarrow \mathbbm{R}^n$, any $B \subseteq W$
such that $\varphi (B)$ is an open ball with finite radius and any element
$\alpha_B \in H_n (M|B ; R)$, define $U (\alpha_B)$ as the set of images of
$\alpha_B$ under the canonical $H_n (M|B ; R) \rightarrow H_n (M|x ; R)$ as
$x$ ranges through $B$. The collection of $U (\alpha_B)$ forms a basis for a
topology on $M_R$. The canonical projection $p : M_R \rightarrow M$ is a covering map.

Recall that a bundle of groups/modules/rings over a space $X$ is defined as a
functor from the fundamental groupoid of $X$ to the category of
groups/modules/rings.

There is a bundle of $R$-modules on $M$, denoted also as $M_R$, such that
$M_R (x) = p^{- 1} (x) = H_n (M|x ; R)$ (with the obvious module structure)
and $M_R ([u]) = L_{u \ast}$ where $u$ is a path in $M$ and $L_u : p^{- 1} (u
(1)) \rightarrow p^{- 1} (u (0))$ is the map defined as in [Hatcher p. 69]
(this construction can also be found in [Whitehead V.1. Example 5]). $M_R$ is
called the ($R$-) orientation bundle of $M$.

\subsection{The Canonical Double Cover}

Let $\widetilde{M} = \{ \pm \mu_x \otimes \tmop{id} \in H_n (M|x ; R)  |  \mu_x \tmop{is} a \tmop{generator} \tmop{of}
H_n (M|x ; \mathbbm{Z}) \}$, here we are using the identification $H_n (M|x ; R) \cong H_n
(M|x ; \mathbbm{Z}) \otimes R$. The manifold $\widetilde{M}$ is oriented as follows:

For each $\mu_x \otimes \tmop{id} \in \widetilde{M}$ where $\mu_x$ is a generator
of $H_n (M|x ; \mathbbm{Z})$, let $B \subseteq M$ be an open ball with finite
radius containing $x$. Let $\mu_B \in H_n (M|B ; \mathbbm{Z})$ be the element
corresponding to $\mu_x$ via the canonical isomorphism $H_n (M|B) \rightarrow
H_n (M|x)$. Then $\mu_B \otimes \tmop{id} \in H_n (M|B ; R) \cong H_n (M|B ;
\mathbbm{Z}) \otimes R$. \\
Let $O_{\mu_x \otimes \tmop{id}}$ be the element
corresponding to $\mu_x \otimes \tmop{id}$ under the identification (each map
below is an isomorphism):
\[ H_n (\widetilde{M} | \mu_x \otimes \tmop{id} ; R) \longleftarrow H_n (U (\mu_B
   \otimes \tmop{id}) | \mu_x \otimes \tmop{id} ; R \nobracket)
   \overset{p_{\ast}}{\longrightarrow} H_n (B|x ; R) \longrightarrow H_n (M|x
   ; R) \]
Then $O_{\mu_x \otimes \tmop{id}}$ is independent of $B$ and $\{ O_{\mu_x
\otimes \tmop{id}} \}$ is an orientaion on $\widetilde{M}$. This orientation will
be referred to as the chosen orientation in what follows.

It is not hard to see that $\widetilde{p} = p_{| \widetilde{M}}$ is a covering space
of index two.

\

Denote by $\tau$ the unique Deck transformation of $\widetilde{M}$ with no fixed
points. Namely, $\tau (y) = - y$ for any $x \in M$ and $y \in H_n (M|x ; R)$.

\begin{lemma}
  The involution $\tau$ reverses the chosen orientation on $\widetilde{M}$.
\end{lemma}

\begin{proof}
  This comes from the commutativity of the diagram{\tmname{}}
 
  \begin{tikzcd}[cramped, column sep=small]
  H_n (\widetilde{M} | \mu_x \otimes \tmop{id} ; R) \arrow[d, "\tau"] &
      H_n (U (\mu_B \otimes \tmop{id}) | \mu_x \otimes \tmop{id} ; R) \arrow[l] \arrow[d, "\tau"] \arrow[r, "p_{\ast}"] &   H_n (B|x ; R) \arrow[r] \arrow[d, equal] &
         H_n (M|x ; R) \arrow[d, equal]\\
       H_n (\widetilde{M} | - \mu_x \otimes \tmop{id} ; R)  &
       H_n (U (- \mu_B \otimes \tmop{id}) | - \mu_x \otimes \tmop{id} ; R) \arrow[l] \arrow[r, "p_{\ast}"]   &  H_n (B|x ; R)  \arrow[r] &  H_n (M|x ; R) 
\end{tikzcd}
  \\
  
\end{proof}

\subsection{The Fundamental Class}

Next we define a fundamental class for $M$.

Let $C_{\ast} (X ; R)$ stands for the singular chain of a space $X$ with
coefficient in $R$. Define
\[ C_n^+ (\widetilde{M} ; R) = \{ \alpha \in C_n (\widetilde{M} ; R) | \tau (\alpha) =
   \alpha \}, C_n^- (\widetilde{M} ; R) = \{ \alpha \in C_n (\widetilde{M} ; R) | \tau
   (\alpha) = - \alpha \} \]

Let $K \subseteq M$ be a compact subspace. Define $\widetilde{K} = \widetilde{p}^{-
1} (K)$ and
\begin{eqnarray*}
  C_n^+ (\widetilde{M} | \widetilde{K} ; R) = C_n^+ (\widetilde{M} ; R) / C_n^+ (\widetilde{M}
  ; R) \cap C_n (\widetilde{M} - \widetilde{K}) &  & \\
  C_n^- (\widetilde{M} | \widetilde{K} ; R) = C_n^- (\widetilde{M} ; R) / C_n^- (\widetilde{M}
  ; R) \cap C_n (\widetilde{M} - \widetilde{K}) &  & 
\end{eqnarray*}
There are exact sequences
\begin{eqnarray*}
  0 \longrightarrow C_n^- (\widetilde{M} | \widetilde{K} ; R) \longhookrightarrow C_n
  (\widetilde{M} | \widetilde{K} ; R) \overset{\Sigma}{\longrightarrow} C_n^+
  (\widetilde{M} | \widetilde{K} ; R) \longrightarrow 0 &  & (1)\\
  0 \longrightarrow C_n^+ (\widetilde{M} | \widetilde{K} ; R) \longhookrightarrow C_n
  (\widetilde{M} | \widetilde{K} ; R) \overset{\Delta}{\longrightarrow} C_n^-
  (\widetilde{M} | \widetilde{K} ; R) \longrightarrow 0 &  & (2)
\end{eqnarray*}
where $\Sigma (\alpha) = \alpha + \tau (\alpha)$ and $\Delta (\alpha) = \alpha
- \tau (\alpha)$.

The chain complex $C_{\ast} (M|K ; R)$ is isomorphic to $C_{\ast}^+ (\widetilde{M}
| \widetilde{K} ; R)$ via $\sigma \otimes r \rightarrow \widetilde{\sigma_1}
\otimes r + \widetilde{\sigma_2} \otimes r$ where $\widetilde{\sigma_1},
\widetilde{\sigma_2}$ are the two liftings of $\sigma : \Delta^n \rightarrow
M$.

By Lemma 3.27(b) of [Ha], $H_k (M|K ; R) = 0, k > n$. Thus (1) produces a
long exact sequence
\begin{eqnarray*}
  \cdots \longrightarrow 0 \longrightarrow H_n (C_{\ast}^- (\widetilde{M} |
  \widetilde{K} ; R)) \longrightarrow H_n (\widetilde{M} | \widetilde{K} ; R)
  \overset{\widetilde{p}_{\ast}}{\longrightarrow} H_n (M|K ; R) \longrightarrow
  \cdots &  & (3)
\end{eqnarray*}

There is a $\mathbbm{Z}_2$ group action on $\widetilde{M}$ such that $\bar{1} \in
\mathbbm{Z}_2$ acts by $\tau$. $\mathbbm{Z}_2$ also acts (as a group) on $R$
with $\bar{1} \cdot r = - r$. Thus one could define $C_n (\widetilde{M} | \widetilde{K} ;
\mathbbm{Z}) \otimes_{\mathbbm{Z}_2} R$ using the routine method of treating a
left $\mathbbm{Z} [\mathbbm{Z}_2]$-module as a right module.

The canonical surjection $C_n (\widetilde{M} | \widetilde{K} ; R) = C_n (\widetilde{M} |
\widetilde{K} ; \mathbbm{Z}) \otimes R \rightarrow C_n (\widetilde{M} | \widetilde{K} ;
\mathbbm{Z}) \otimes_{\mathbbm{Z}_2} R$ has the same kernel as $\Delta$ in
(2): $C_n^+ (\widetilde{M} | \widetilde{K} ; R)$. Thus we obtain an identification
\[ C_n^- (\widetilde{M} | \widetilde{K} ; R) \longleftrightarrow C_n (\widetilde{M} |
   \widetilde{K} ; \mathbbm{Z}) \otimes_{\mathbbm{Z}_2} R \]
On the other hand, let $C_{\ast} (M|K ; M_R)$ be the chain complex of $(M, M -
K)$ with coefficient in the bundle $M_R$. We can define a homomorphism
\[ \phi : C_n (\widetilde{M} | \widetilde{K} ; \mathbbm{Z}) \otimes_{\mathbbm{Z}_2} R
   \longrightarrow C_{\ast} (M|K ; M_R), \phi (\widetilde{\sigma}
   \otimes_{\mathbbm{Z}_2} r) = (r \widetilde{\sigma} (e_0)) \widetilde{p} \circ
   \widetilde{\sigma} \]
where $\widetilde{\sigma} : \Delta^n \rightarrow \widetilde{M}$ and $(r \widetilde{\sigma}
(e_0)) \widetilde{p} \circ \widetilde{\sigma}$ is represented by the element of
$C_{\ast} (M ; M_R) = \underset{\sigma : \Delta^n \rightarrow M}{\oplus} M_R
(\sigma (e_0))$ with $r \widetilde{\sigma} (e_0)$ on the $\widetilde{p} \circ
\widetilde{\sigma}$ coordinate and 0 otherwise. It is not hard to verify that
$\phi$ is an isomorphism. Hence we obtain an identification $C_{\ast}^-
(\widetilde{M} | \widetilde{K} ; R) \leftrightarrow C_{\ast} (M|K ; M_R)$. Explicitly
this comes from the diagram
\\

\begin{tikzcd}
  C_{\ast}^- (\widetilde{M} | \widetilde{K} ; R) &
  C_{\ast} (\widetilde{M} | \widetilde{K} ; R) = C_{\ast} (\widetilde{M} | \widetilde{K} ;
  \mathbbm{Z}) \otimes R \arrow[l,"\Delta"] \arrow[d] \\
   & C_{\ast} (\widetilde{M} | \widetilde{K} ;\mathbbm{Z}) \otimes_{\mathbbm{Z}_2} R \arrow[d,"\phi"] &(4) \\
   & C_{\ast} (M|K ; M_R)   
\end{tikzcd}
\\

The above identification is a chain isomorphism, thus we have
\[ H_n (C_{\ast}^- (\widetilde{M} | \widetilde{K} ; R)) \longleftrightarrow H_n (M|K ;
   M_R) \]
where $H_n (M|K ; M_R)$ is the homology group with coefficients in $M_R$.

Plug this into (3), we get
\begin{eqnarray*}
  \cdots \longrightarrow 0 \longrightarrow H_n (M|K ; M_R) \longrightarrow H_n
  (\widetilde{M} | \widetilde{K} ; R) \overset{\widetilde{p}_{\ast}}{\longrightarrow} H_n
  (M|K ; R) \longrightarrow \cdots &  & (5)
\end{eqnarray*}
Since (4) is natural with respect to compact subspaces $K_1 \subseteq K_2
\subseteq M$, so is (5).

The chosen orientation on $\widetilde{M}$ uniquely determines an element in $\nu_K
\in H_n (\widetilde{M} | \widetilde{K} ; R)$ that restrict to the orientation at each
$\widetilde{x} \in \widetilde{K}$ (cf. Lemma 3.27 (a) of [Ha]).

\begin{lemma}
  $\widetilde{p}_{\ast} (\nu_K) = 0$ for any compact subspace $K \subseteq M$.
\end{lemma}

\begin{proof}
  By the uniqueness part of Lemma 3.27(a) of [Ha], it suffice to prove
  $\widetilde{p}_{\ast} (\nu_K)$ restrict to $0$ at each $x \in M$. Let
  $\widetilde{p}^{- 1} (x) = \{ \widetilde{x}, \widetilde{x}' \}$. We have a commutative
  diagram

\begin{center}
\begin{tikzcd}
       H_n (\widetilde{M} | \widetilde{K} ; R)\arrow[r] \arrow[d,"\widetilde{p}_{\ast}"]  &  H_n (\widetilde{M}
      | \{ \widetilde{x}, \widetilde{x}' \} ; R) \arrow[d,"\widetilde{p}_{\ast}"] \\
       H_n (M|K ; R)  \arrow[r]  &  H_n (M|x ; R) 
\end{tikzcd}
\end{center}
  
  where horizontal maps are induced by inclusions. Hence the goal become
  showing that $\nu_K$ maps to $0$ via $H_n (\widetilde{M} | \widetilde{K} ; R)
  \rightarrow H_n (\widetilde{M} | \{ \widetilde{x}, \widetilde{x} \} ; R) \rightarrow H_n
  (M|x ; R)$.
  
  Take an open neighborhoods $U$ of $x$ such that \ $\widetilde{p}^{- 1} (U) =
  \widetilde{U} \sqcup \widetilde{U}'$ and $\widetilde{x} \in \widetilde{U}, \widetilde{x}' \in
  \widetilde{U}'$. Consider the following commutative diagrams\\

 \begin{tikzcd}
      &    H_n (\widetilde{M} | \widetilde{x} ; R) \oplus H_n (\widetilde{M} |
      \widetilde{x}' ; R) \\
       H_n (\widetilde{M} | \widetilde{K} ; R)  \arrow[r,"k_{\ast}"]&  H_n (\widetilde{M} | \{ \widetilde{x},
      \widetilde{x}' \} ; R) \arrow[u,"j_{\ast} \oplus j'_{\ast}"] \arrow[r,"\widetilde{p}_{\ast}"] & H_n (M|x ; R) \\
      &    H_n (\widetilde{U} | \widetilde{x} ; R) \oplus H_n (\widetilde{U}' |
      \widetilde{x}' ; R) \arrow[u,"\iota_{\ast} \oplus \iota'_{\ast}"] \arrow[r,"\widetilde{p}_{\ast}"] & 
       H_n (U|x ; R) \arrow[u, "i_{\ast}"] 
\end{tikzcd}

\begin{center}
  
 \begin{tikzcd}
       H_n (\widetilde{M} | \widetilde{x} ; R) \arrow[r,"\tau_{\ast}"]   &  H_n (\widetilde{M} | \widetilde{x}'
      ; R) \\
       H_n (\widetilde{U} | \widetilde{x} ; R) \arrow[u, "j_{\ast} \circ \iota_{\ast}"] \arrow[r,"\tau_{\ast}"] &  H_n (\widetilde{U}' | \widetilde{x}'; R) \arrow[u,"j_{\ast}' \circ \iota_{\ast}'"]
\end{tikzcd}
\end{center}
  
  where $\iota, \iota', j, j', i, k$ are inclusions. By excision and
  additivity, $\iota_{\ast} \oplus \iota'_{\ast}$ is an isomorphism. On the
  other hand, it is not hard to observe that $(j_{\ast} \oplus j'_{\ast})
  \circ (\iota_{\ast} \oplus \iota'_{\ast}) = (j_{\ast} \circ \iota_{\ast})
  \oplus (j_{\ast}' \circ \iota_{\ast}')$. Obviously $i_{\ast}$ is also an
  isomorphism.
  
  For any $\alpha \in H_n (\widetilde{M} | \widetilde{K} ; R)$, $k_{\ast} (\alpha) =
  \iota_{\ast} (\beta) + \iota'_{\ast} (\gamma)$. Hence $(j_{\ast} \oplus
  j'_{\ast}) \circ k_{\ast} (\alpha) = \iota_{\ast} \circ j_{\ast} (\beta)
  \oplus \iota_{\ast}' \circ j_{\ast}' (\gamma)$. Note that $j_{\ast} \circ
  k_{\ast} (\alpha)$ and $j_{\ast}' \circ k_{\ast}' (\alpha)$ are orientations
  at $\widetilde{x}, \widetilde{x}'$ respectively. So $\tau_{\ast} \circ j_{\ast}
  \circ k_{\ast} (\alpha) = - j_{\ast}' \circ k_{\ast}' (\alpha)$. This
  implies $\tau_{\ast} (\beta) = - \gamma$. But $\widetilde{p} \circ \tau =
  \widetilde{p}$, whence $\widetilde{p}_{\ast} (\beta, \gamma) = \widetilde{p}_{\ast}
  (\beta) + \widetilde{p}_{\ast} (\gamma) = \widetilde{p}_{\ast} \circ \tau_{\ast}
  (\beta) + \widetilde{p}_{\ast} (\gamma) = \widetilde{p}_{\ast} (- \gamma) +
  \widetilde{p}_{\ast} (\gamma) = 0$. Thus $\widetilde{p}_{\ast} \circ k_{\ast}
  (\alpha) = \widetilde{p}_{\ast} \circ (\iota_{\ast} \oplus \iota'_{\ast})
  (\beta, \gamma) = i_{\ast} \circ \widetilde{p}_{\ast} (\beta, \gamma) = 0$.
\end{proof}

\

By Lemma 2 and the exactness of (5), the orientation of $\widetilde{M}$ uniquely
determines an element of $H_n (M|K ; M_R)$, denoted also as $\nu_K$. The
naturality of (4) show that $\{ \nu_K |K \subseteq M \tmop{compact} \}$ is
compatible with respect to inclusion of $K$'s and thus define an element $[M]$
of $\underset{K}{\underset{\longleftarrow}{\lim}} H_n (M|K ; M_R)$ where the
inverse limit is taken with respect to all compact $K \subseteq M$ and
inclusions $K_1 \subseteq K_2 \subseteq M$. $[M]$ is called the fundamental
class of $M$.

\

\

\subsection{Restricting to open subspaces; Compatibility}

Let $U \subseteq M$ be an open subset. Here is a few elementary facts we shall
need:

\begin{proposition}
  \
  
  i)There is a canonical embedding $\widetilde{U} \hookrightarrow \widetilde{M}$
  induced by the excision $H_n (U|x ; R) \rightarrow H_n (M|x ; R)_{}, x \in
  U$.
  
  ii)The bundle $U_R$ (of $R$-modules) is canonically isomorphic to the
  restriction of the bundle $M_R$ to $U$.
  
  iii)The chosen orientation of $\widetilde{M}$ restrict to the chosen orientation
  on $\widetilde{U}$.
  
  iv)For any $K \subseteq U$ compact, the excision
  $H\tmrsub{n}(\widetilde{U}|\widetilde{K};R){\rightarrow}H\tmrsub{n}(\widetilde{M}|\widetilde{K}$;R){}
  sends $\nu_K^U$ to $\nu_K^M$ where $\nu_K^U$, $\nu_K^M$ are elements
  determined by the chosen orientation.
  
  v)There diagram
  
\begin{center}
 \begin{tikzcd}
       H_n (M|K ; M_R)  \arrow[r] &
       H_n(\widetilde{M}|\widetilde{K};R) \\
       H_n (U|K ; M_R) \arrow[r] \arrow[u]  &
       H_n(\widetilde{U}|\widetilde{K};R) \arrow[u]
\end{tikzcd}
\end{center}
  
  commutes, where horizontal maps are those in (5) and vertical ones are
  induced by inclusion.
\end{proposition}

\begin{corollary}
  The homomorphism $H_n (U|K ; U_R) \longrightarrow H_n (M|K ; M_R)$ induced
  by inclusion sends $\nu_K^U$ to $\nu_K^M$, where $\{ \nu_K^U \}, \{ \nu_K^M
  \}$ define $[U], [M]$ respectively.
\end{corollary}

Thus $\nu_K^U$ is compatible with respect to inclusions of open $U$'s as well
as compact $K$'s.

\subsection{Cap Products}

We start with defining tensor product of bundle of modules. Let $G$ (resp.
$G'$) be a bundle of left (resp. right) $R$-modules over a space $X$. The
tensor product $G \otimes_R G'$ is defined as the bundle of abelian groups
where $G \otimes_R G' (x) = G (x) \otimes_R G' (x)$ and $G \otimes_R G' ([u])
= G ([u]) \otimes_R G' ([u])$ for any $x \in X$ and $u : I \rightarrow X$.

Denote the vertices of $\Delta^n$ as $e_0, e_1, \cdots, e_n$. Let $\sigma :
\Delta^n \rightarrow X$ be a continuous map. For $0 \leqslant i_1 < i_2 <
\cdots < i_k \leqslant n$, let $\sigma_{[i_1, i_2, \cdots, i_k]}$ denote
$\sigma$ restricted to the simplex $e_{i_1} e_{i_2} \cdots e_{i_k}$.

Now we are able to define the cap product on (absolute) chains. Assume that
$G$ (resp. $G'$) is a bundle of left (resp. right) $R$-modules over a space
$X$, the cap product is defined as
\begin{eqnarray*}
  C^k (X ; G) \otimes_R C_n (X ; G') & \overset{\frown}{\longrightarrow} &
  C_{n - k} (X ; G \otimes_R G')\\
  c \otimes g \sigma & \longrightarrow & (G (\sigma_{[0, n - k]}) (c
  (\sigma_{[n - k, \cdots, n]})) \otimes g) \sigma_{[0, \cdots, n - k]}
\end{eqnarray*}
where $c \in C^k (X ; G) = \underset{\rho : \Delta^k \rightarrow X}{\Pi} G
(\rho (e_0))$, $g \in G (\sigma (e_0))$ and $g \sigma \in C_n (X ; G') =
\underset{\eta : \Delta^n \rightarrow X}{\oplus} G' (\eta (e_0))$ denotes the
element which has value $g$ on the $\sigma$-coordinate and 0 otherwise.

If $A_1, A_2$ are subspaces of $X$, the above absolute cap product induces a
relative
\[ C^k (X, A_1 ; G) \otimes_R C_n (X, A_1 + A_2 ; G')
   \overset{\frown}{\longrightarrow} C_{n - k} (X, A_2 ; G \otimes_R G') \]
where the relative $C_{\ast}, C^{\ast}$ are defined in the obvious way.

\

The cap product satisfies the identity
\[ \partial (c \frown \alpha) = c \frown (\partial \alpha) - (\delta c) \frown
   \alpha, c \in C^k (X ; G), \alpha \in C_n (X ; G') \]
Note that the sign appearing in the above equation is a result of our adopting
the definition in [Wh].

There is thus an induced cap product on (co)homology
\[ H^k (X, A_1 ; G) \otimes_R H_n (X, A_1 + A_2 ; G')
   \overset{\frown}{\longrightarrow} H_{n - k} (X, A_2 ; G \otimes_R G') \]
For the special case where $X = M, A_1 = M - K, A_2 = \varnothing$ and $G' =
M_R$, we obtain
\begin{eqnarray*}
  H^k (M|K ; G) \otimes_R H_n (M|K ; M_R) \overset{\frown}{\longrightarrow}
  H_{n - k} (M ; G \otimes_R M_R) &  & (5)
\end{eqnarray*}

Naturality with respect to inclusion of compact subspaces can be easily
verified, thus (5) produces
\[ \underset{K}{\underset{\longrightarrow}{\lim}} H^k (M|K ; G) \otimes_R
   \underset{K}{\underset{\longleftarrow}{\lim}} H_n (M|K ; M_R)
   \overset{\frown}{\longrightarrow} H_{n - k} (M ; G \otimes_R M_R) \]
Note that $\underset{K}{\underset{\longrightarrow}{\lim}} H^k (M|K ; G)$ is
canonically isomorphic to $H^k_c (M ; G)$, the cohomology of $M$ with compact
support and with coefficient in $G$ (the proof of this in the case of ordinary
coefficients can be found in [Ha] Section 3.3). So the above becomes
\[ H^k_c (M ; G) \otimes_R \underset{K}{\underset{\longleftarrow}{\lim}} H_n
   (M|K ; M_R) \overset{\frown}{\longrightarrow} H_{n - k} (M ; G \otimes_R
   M_R) \]
If one choose the fundamental class $[M] \in
\underset{K}{\underset{\longleftarrow}{\lim}} H_n (M|K ; M_R)$, there is a
homomorphism
\[ H^k_c (M ; G) \overset{\frown [M]}{\longrightarrow} H_{n - k} (M ; G
   \otimes_R M_R) \]

\section{Proof of the Duality Theorem}

\begin{lemma}
  Let $U, V$ be open subsets of $M$ with $U \cup V = M$. $K \subseteq U, L
  \subseteq V$ are compact subspaces. Let $G$ be a bundle of right $R$-modules
  over $M$. Then the following diagram commutes up to sign
  
\begin{flushleft}
 \begin{tikzcd}[cramped, column sep=tiny]
      \  \arrow[r,"\delta"]&  H^k (M|K \cap L ; G)  \arrow[r] \arrow[d] & H^k (M|K ; G) \oplus H^k (M|L ; G) \arrow[r] \arrow[d] & H^k (M|K \cup L ; G)  \arrow[dd, "\frown \nu_{K \cup L}^M"]  \ \\
      (6)  &     H^k (U \cap V|K \cap L ; G) \arrow[d,"\frown \nu_{K \cap L}^{U \cap V}"]  &  H^k (U|K ; G) \oplus H^k (V|L ;     G) \arrow[d, "(\frown \nu_K^U) \oplus (\frown - \nu_L^V)"]  \\
     \ \arrow[r,"\partial"]&  H_{n - k} (U \cap V ; G \otimes_R M_R)  \arrow[r]   &  H_{n - k} (U ; G \otimes_R M_R) \oplus H_{n - k} (V ; G \otimes_R M_R) \arrow[r]   &  H_{n - k} (M ; G \otimes_R M_R) \  \\
\end{tikzcd}

\end{flushleft}
  
  where the two rows are Mayer-Vietoris sequences (see Appendix), the upper
  left and upper middle maps are induced by inclusions.
\end{lemma}

\begin{proof}
  We begin with the two blocks without $\delta$ or $\partial$. Commutativity
  of the one on the left would follow once we establish the commutativity of
  the following
  
\begin{flushleft}
\begin{tikzcd}[column sep=scriptsize]
       H^k (M|K \cap L ; G) \arrow[r] \arrow[dd] \arrow[rd] &  H^k (M|K ; G) \arrow[r] & H^k (U|K ; G) \arrow[dd, "\frown \nu_K^U"] \\
      &  H^k (U|K \cap L ; G) \arrow[ru] \arrow[rd, "\frown \nu_{K \cap L}^{U}"]  \\
       H^k (U \cap V|K \cap L ; G) \arrow[ru] \arrow[r,"\frown \nu_{K \cap L}^{U \cap V}"']  &  H_{n - k} (U \cap V ; G \otimes_R (U \cap V)_R) \arrow[r] &   H_{n - k} (U ; G \otimes_R U_R) 
\end{tikzcd}
\end{flushleft}
  
  in which all arrows that is not a cap product is induced by inclusion.
  
  The above diagram commutes because of the compatibility of $\nu_K^U$'s.
  
  Similarly, commutativity of the block on the right in (6) follows from
  commutativity of\\

\begin{tikzcd}
       H^k (M|L ; G) \arrow[r] \arrow[d] \arrow[rd,"\frown \nu_{L}^M"']   &  H^k (M|K \cup L ; G) \arrow[d, "\frown \nu_{K \cup L}^M"] \\
       H^k (V|L ; G) \arrow[d, "\frown \nu_L^V"]   &  H_{n - k} (M ; G \otimes_R M_R) \\
       H_{n - k} (V ; G \otimes_R V_R) \arrow[ru] 
\end{tikzcd}
\\

  Again such commutativity comes from compatibility of $\nu_K^U$'s.
  
  For the block involving $\partial$ and $\delta$, [Ha] presented a
  detailed proof (p. 246-247) which carries verbatim to the case of twisted
  coefficients. So I shall not rewrite it in this note. Due to the difference
  in the convention of signs defining $\delta, \partial$, the block commutes
  up to a factor of $- 1$ instead of $(- 1)^{k + 1}$.
\end{proof}

\begin{corollary}
  Let $U, V$ be open subsets of $M$ with $U \cup V = M$. Let $G$ be a bundle
  of right $R$-modules over $M$. Then there is a (up to sign) commutative
  diagram

\begin{flushleft}
\begin{tikzcd}[cramped, column sep=small, row sep=0.2em]
   \  \arrow[r] &  H^k_c (U \cap V ; G) \arrow[r] \arrow[dddddd]  &  H^k_c
      (U ; G) \oplus H^k_c (V ; G) \arrow[r] \arrow[dddddd] &  H^k_c (M ; G) \arrow[r] \arrow[dddddd]  & \   \\
      & \\
      & \\
      & \\
      & \\
      & \\
   \ \arrow[r] &  H_{n - k} (U \cap V ; G \otimes_R M_R)  \arrow[r] &  H_{n - k} (U ; G \otimes_R M_R)  \arrow[r] &   H_{n - k} (M ; G \otimes_R M_R) \arrow[r] & \ \\
     &   & \oplus H_{n - k} (V ; G \otimes_R M_R) 
\end{tikzcd}
\end{flushleft}

\begin{flushleft}
where vertical maps are cap products with respective fundamental classes and the two rows are Mayer-Vietoris sequences .
\end{flushleft}
\end{corollary}

\begin{proof}
  This follows from the preceding Lemma by taking the direct limit of (6) with
  respect to the directed set $\{ (K, L) |K \subseteq U \tmop{exact}, L
  \subseteq V, (K, L) \leqslant (K', L') \tmop{iff} K \subseteq K', L
  \subseteq L' \}$. 
\end{proof}

Now we can prove the Poincare Duality.

\begin{theorem}
  For any manifold $M$ and any bundle of right $R$-modules $G$, the
  homomorphism
  \[ H^k_c (M ; G) \overset{\frown [M]}{\longrightarrow} H_{n - k} (M ; G
     \otimes_R M_R) \]
  is an isomorphism.
\end{theorem}

\begin{proof}
  The proof of [Ha] Theorem 3.35 applies with one exception. In the case
  when $M =\mathbbm{R}^n$, one use the fact that $\mathbbm{R}^n$ is
  contractible to deduce that $G$ and $(\mathbbm{R}^n)_R$ is isomorphic to a
  constant bundle. Choose and fix a base point of $\mathbbm{R}^n$, say $0$.
  Define $G_0 = G (0)$. Identify $(\mathbbm{R}^n)_R (0)$ with $R$. Then $H_c^k
  (\mathbbm{R}^n ; G)$, $\underset{K}{\underset{\longleftarrow}{\lim}} H_n
  (\mathbbm{R}^n |K ; M_R)$ and $H_{n - k} (\mathbbm{R}^n ; G \otimes_R
  (\mathbbm{R}^n)_R)$ can be canonically identified with $H_c^k (\mathbbm{R}^n
  ; G_0)$, $\underset{K}{\underset{\longleftarrow}{\lim}} H_n (\mathbbm{R}^n
  |K ; R)$ and $H_{n - k} (\mathbbm{R}^n ; G_0 \otimes_R R)$ respectively.
  Such identifications are compatible with the cap product. Under such
  identification, it is not hard to check from the definition that $[M] =
  [\mathbbm{R}^n] \in \underset{K}{\underset{\longleftarrow}{\lim}} H_n
  (\mathbbm{R}^n |K ; M_R)$ as defined in this note is an fundamental class in
  $\underset{K}{\underset{\longleftarrow}{\lim}} H_n (\mathbbm{R}^n |K ; R)$,
  i.e. an element that restrict to the local orientation at each $x \in
  \mathbbm{R}^n$ for a chosen orientation. Thus the Poincare Duality for the
  $R$-orientable manifold $\mathbbm{R}^n$ (the one we apply here is slightly more general than what appears in [Ha] since the coefficients can be in any $R$-module, but the same proof as in the simpler case carries verbatim to prove this generalized case) proves that $\frown [M]$ is an
  isomorphism.
\end{proof}

\

\
It should be mentioned that our result concerns merely cohomology with compact support. There is a version of Poincare duality for ordinary cohomology as recorded in [Sp], which is an isomorphism between Alexander cohomology and locally finite homology. The proof of this result, however, seems to require sheaf theory or some equally sophisticated machinery.
\\
\\

Appendix: The (Relative) Mayer Vietoris sequences with local coefficients

\

We shall need the following lemma, whose proof is essentially identical to
that of corresponding results for (co)homology with constant coefficients:

\begin{lemma}
  Let $X = \underset{\alpha}{\cup} \tmop{Int} X_{\alpha}$ where $X_{\alpha}$'s
  are subspaces. Let $G$ be a bundle of groups on $X$. Define $C_{\ast} \left(
  \underset{\alpha}{\Sigma} X_{\alpha} ; G \right) = \left\{
  \underset{i}{\Sigma} g_i \sigma_i \in C_{\ast} (X ; G) | \tmop{each}
  \sigma_i (\Delta^{\ast}) \tmop{is} \tmop{contained} \tmop{in} \tmop{some}
  X_{\alpha} \right\}$. Define $C^{\ast} \left( \underset{\alpha}{\Sigma}
  X_{\alpha} ; G \right) = \underset{\sigma : \Delta^{\ast} \rightarrow X,
  \sigma (\Delta^{\ast}) \subseteq X_{\alpha} \tmop{for} \tmop{some}
  \alpha}{\Pi} \sigma (e_0)$. Then the canonical $C_{\ast} \left(
  \underset{\alpha}{\Sigma} X_{\alpha} ; G \right) \rightarrow C_{\ast} (X ;
  G)$ and $C^{\ast} (X ; G) \rightarrow C^{\ast} \left(
  \underset{\alpha}{\Sigma} X_{\alpha} ; G \right)$ induce isomorphism on
  homology.
\end{lemma}

\

Given a pair $(X, Y) = (A \cup B, C \cup D)$ with $C \subseteq A, D \subseteq
B$ and $X = \tmop{Int}_X A \cup \tmop{Int}_X B, Y = \tmop{Int}_Y C \cup
\tmop{Int}_Y D$. For a bundle of groups $G$ on $X$, there are Mayer-Vietoris
sequences:
\[  \longrightarrow H_n (A \cap B,
   C \cap D ; G) \longrightarrow H_n (A, C ; G) \oplus H_n (B, D ; G)
   \longrightarrow H_n (X, Y ; G) \longrightarrow  \]
and
\[ \longrightarrow H^n (X, Y ; G) \longrightarrow H^n (A, C ; G) \oplus
   H^n (B, D ; G) \longrightarrow H^n (A \cap B, C \cap D ; G) \longrightarrow
   \]
The sequence for homology is deduced by essentiallly the same way as ordinary
(untwisted) coefficients (cf. Hatcher). On the other hand, the proof for the
cohomological Mayer-Vietoris sequence with untwisted coefficients (cf Hathcer
pp. 204) almost carries to the twisted case except when proving
\[ 0 \longrightarrow C^n (A + B, C + D ; G) \overset{\psi}{\longrightarrow}
   C^n (A, C ; G) \oplus C^n (B, D ; G) \overset{\varphi}{\longrightarrow} C^n
   (A \cap B, C \cap D ; G) \longrightarrow 0 \]
is exact. This sequence no longer comes from dualizing the corresponding
sequence for homology. Thus one has to prove the exactness by hand. The
non-trivial part is proving the surjectivity of $\varphi$. We will show this
by constructing for any $\alpha \in C^n (A \cap B, C \cap D ; G)$ a pair
$(\beta, \gamma) \in C^n (A, C ; G) \oplus C^n (B, D ; G)$ as follows:
\[ \beta (\sigma) = \left\{ \begin{array}{l}
     0 \quad \sigma (\Delta^n) \nsubseteq A \cap B \quad \tmop{or} \quad
     \sigma (\Delta^n) \subseteq C\\
     \alpha (\sigma) \quad \tmop{otherwise}
   \end{array} \right. \]
\[ \gamma (\sigma) = \left\{ \begin{array}{l}
     - \alpha (\sigma) \quad \sigma (\Delta^n) \subseteq C\\
     0 \hspace{4em} \tmop{otherwise}
   \end{array} \right. \]
It is not hard to verify that $\varphi (\beta, \gamma) = \alpha$. Thus
$\varphi$ is surjective.

\section{Acknowledgement}

The author would like to thank Prof. Kwasik for his generous help during the time of writing.

\section{Bibliography}
\begin{flushleft}
[Ha] A. Hatcher, Algebraic Topology(2002)

[Sp] E. Spanier, Algebraic Topology(1966), Springer-Verlag\\

[Wh] G. Whitehead, Elements of Homotopy Theory(1978), Springer-Verlag
\end{flushleft}

\end{document}